\documentclass[12pt]{amsart}

\usepackage{amssymb,euscript,amsmath, mathrsfs,amscd}
\usepackage[dvips]{graphicx}
\usepackage{epsfig,pstricks}
\usepackage{fullpage}

\newcounter{ENUM}
\newcommand{\be}{\begin{enumerate}}
\newcommand{\ee}{\end{enumerate}}
\newcommand{\beas}{\begin{eqnarray*}}
\newcommand{\eeas}{\end{eqnarray*}}
\newcommand{\bea}{\begin{eqnarray}}
\newcommand{\eea}{\end{eqnarray}}
\newcommand{\beq}{\begin{equation}}
\newcommand{\eeq}{\end{equation}}

\newcommand{\st}{\,:\,}
\newcommand{\eqdef}{\mathrel{\mathop:}=}

\newcommand{\jtl}{\mathrm{JT}_\lambda}

\newcommand{\jtq}{\mathrm{JT}_\lambda(q)}
\newcommand{\la}{\lambda}
\newcommand{\ca}{\mathcal{A}}
\newcommand{\cb}{\mathcal{B}}
\newcommand{\calr}{\mathcal{R}}
\newcommand{\sn}{\mathfrak{S}_n}

\newtheorem{thm}{Theorem}[section]

\newtheorem{conj}[thm]{Conjecture}

\theoremstyle{definition}

\newtheorem{ex}[thm]{Example}

\theoremstyle{remark}

\numberwithin{equation}{section}

\newcommand{\bm}[1]{{\boldsymbol{#1}}}

\def\cb{\mathcal B}

\def\rr{\mathbb{R}}
\def\zz{\mathbb{Z}}
\def\nn{\mathbb{N}}

\newcommand{\qq}{\mathbb{Q}}
\newcommand{\diag}{\mathrm{diag}}

\newcommand{\snf}{\stackrel{\mathrm{snf}}{\rightarrow}}
\newcommand{\coker}{\mathrm{coker}}


\begin{document}
\title{Smith Normal Form in Combinatorics}

\date{\today}

\author{Richard P. Stanley}
\email{rstan@math.mit.edu}
\address{Department of Mathematics, University of Miami, Coral Gables,
FL 33124}

\thanks{Partially supported by NSF grant DMS-1068625.}

\begin{abstract}
This paper surveys some combinatorial aspects of Smith normal form,
and more generally, diagonal form.  The discussion includes general
algebraic properties and interpretations of Smith normal form,
critical groups of graphs, and Smith normal form of random integer
matrices. We then give some examples of Smith normal form and diagonal
form arising from (1) symmetric functions, (2) a result of Carlitz, Roselle,
and Scoville, and (3) the Varchenko matrix of a hyperplane arrangement.
\end{abstract}

\maketitle


\section{Introduction}
Let $A$ be an $m\times n$ matrix over a field $K$. By means of
elementary row and column operations, namely: 
 \be\item
 add a multiple of a row
(respectively, column) to another row (respectively, column), or
  \item multiply a row or column by a unit (nonzero element) of $K$,
 \ee
 we can transform $A$ into a matrix that vanishes off the main diagonal
(so $A$ is a diagonal matrix if $m=n$) and whose main diagonal
consists of $k$ 1's followed by $m-k$ 0's. Moreover, $k$ is uniquely
determined by $A$ since $k=\mathrm{rank}(A)$.

What happens if we replace $K$ by another ring $R$ (which we always
assume to be commutative with identity 1)? We allow the same row and
column operations as before. Condition (2) above is ambiguous since a
unit of $R$ is not the same as a nonzero element. We want the former
interpretation, i.e., we can multiply a row or column by a unit
only. Equivalently, we transform $A$ into a matrix of the form $PAQ$,
where $P$ is an $m\times m$ matrix and $Q$ is an $n\times n$ matrix,
both invertible over $R$. In other words, $\det P$ and $\det Q$ are
units in $R$. Now the situation becomes much more complicated.

We say that $PAQ$ is a \emph{diagonal form} of $A$ if it vanishes off
the main diagonal. (Do not confuse the diagonal form of a square
matrix with the matrix $D$ obtained by diagonalizing $A$. Here
$D=XAX^{-1}$ for some invertible matrix $X$, and the diagonal entries
are the eigenvalues of $A$.) If $A$ has a diagonal form $B$ whose main
diagonal is $(\alpha_1,\dots,\alpha_r,0,\dots,0)$, where $\alpha_i$
divides $\alpha_{i+1}$ in $R$ for $1\leq i\leq r-1$, then we call $B$
a \emph{Smith normal form} (SNF) of $A$. If $A$ is a nonsingular
square matrix, then taking determinants of both sides of the equation
$PAQ=B$ shows that $\det A=u\alpha_1\cdots \alpha_n$ for some unit
$u\in R$. Hence an SNF of
$A$ yields a factorization of $\det A$. Since there is a huge
literature on determinants of combinatorially interesting matrices
(e.g., \cite{kratt1}\cite{kratt2}), finding an SNF of such matrices
could be a fruitful endeavor.

In the next section we review the basic properties of SNF, including
questions of existence and uniqueness, and some algebraic aspects. In
Section~\ref{sec:cgp} we discuss connections between SNF and the
abelian sandpile or chip-firing process on a graph.  The distribution
of the SNF of a random integer matrix is the topic of
Section~\ref{sec:random}. The remaining sections deal with some
examples and open problems related to the SNF of combinatorially
defined matrices.

We will state most of our results with no proof or just the
hint of a proof. It would take a much longer paper to summarize all
the work that has been done on computing SNF for special matrices. We
therefore will sample some of this work based on our own interests and
research. We will include a number of open problems which we hope will
stir up some further interest in this topic. 

\section{Basic properties}
In this section we summarize without proof the
basic properties of SNF. We will use the following notation. If $A$ is
an $m\times n$ matrix over a ring $R$, and $B$ is the matrix with
$(\alpha_1,\dots,\alpha_m)$ on the main diagonal and 0's elsewhere
then we write 
$A\snf(\alpha_1,\dots,\alpha_m)$ to indicate that $B$ is an SNF of
$A$. 

\subsection{Existence and uniqueness}  For 
connections with combinatorics we are primarily interested in the ring
$\zz$ or in polynomial rings over a field or over $\zz$. However, it
is still interesting to ask over what rings $R$ does a matrix always
have an SNF, and how unique is the SNF when it exists. For this
purpose, define an \emph{elementary divisor ring} $R$ to be a ring
over which every matrix has an SNF. Also define a
\emph{B\'ezout ring} to be a commutative ring for
which every finitely generated ideal is principal. Note that a
noetherian B\'ezout ring is (by definition) a principal ideal ring,
i.e., a ring (not necessarily an integral domain) for which every
ideal is principal. An important example of a principal ideal ring
that is not a domain is $\zz/k\zz$ (when $k$ is not prime). Two
examples of non-noetherian B\'ezout domains are the ring of entire
functions and the ring of all algebraic integers. 

\begin{thm} \label{thm:rings}
Let $R$ be a commutative ring with identity.
\be\item If every rectangular matrix over $R$ has an SNF, then $R$ is
a B\'ezout ring. In fact, if $I$ is an ideal with a minimum size
generating set $a_1,\dots,a_k$, then the $1\times 2$ matrix
$[a_1,a_2]$ does not have an SNF. See \cite[p.~465]{kap}.
 \item Every diagonal matrix over $R$ has an SNF if and only if $R$ is
   a B\'ezout ring \cite[(3.1)]{l-l-s}. 
 \item A B\'ezout domain $R$ is an elementary divisor domain if and only
   if it satisfies:
  $$ \mbox{For all $a,b,c\in R$ with $(a,b,c)=R$, there exists $p,q\in
     R$ such that $(pa,pb+qc)=R$.} $$
  See \cite[{\textsection}5.2]{kap}\cite[{\textsection}6.3]{f-s}.
  \item Every principal ideal ring is an elementary divisor ring. This
    is the classical existence result (at least for principal ideal
    domains), going back to Smith \cite{smith} for the integers.
  \item Suppose that $R$ is an \emph{associate ring}, that is, if two
    elements $a$ and $b$ generate the same principal ideal there is a
    unit $u$ such that $ua=b$. (Every integral domain is an associate
    ring.) If a matrix $A$ has an SNF $PAQ$ over $R$, then $PAQ$ is
    unique (up to multiplication of each diagonal entry by a
    unit). This result is immediate from \cite[{\textsection}IV.5,
      Thm.~5.1]{l-q}.  
\ee
\end{thm}

It is open whether every B\'ezout domain is an elementary divisor
domain. For a recent paper on this question, see Lorenzini
\cite{lorenzini}. 

Let us give a simple example where SNF does not exist.

\begin{ex}
Let $R=\zz[x]$, the polynomial ring in one variable over $\zz$, and
let $A=\left[ \begin{array}{cc} 2 & 0\\ 0 & x \end{array}
  \right]$. Clearly $A$ has a diagonal form (over $R$) since it is
already a diagonal matrix. Suppose that $A$ has an SNF $B=PAQ$. The
only possible SNF (up to units $\pm 1$) is $\diag(1,2x)$, since $\det
B = \pm 2x$. Setting $x=2$ in $B=PAQ$ yields the SNF $\diag(1,4)$
over $\zz$, but setting $x=2$ in $A$ yields the SNF $\diag(2,2)$.
\end{ex}

Let us remark that there is a large literature on the computation of
SNF over a PID (or sometimes more general rings) which we will not
discuss. We are unaware of any literature on deciding whether a given
matrix over a more general ring, such as $\qq[x_1,\dots,x_n]$ or
$\zz[x_1,\dots,x_n]$, has an SNF.

\subsection{Algebraic interpretation}
Smith normal form, or more generally diagonal form, has a simple
algebraic interpretation. Suppose that the $m\times n$ matrix $A$ over
the ring $R$ has a diagonal form with diagonal entries
$\alpha_1,\dots,\alpha_m$. The rows $v_1,\dots,v_m$ of $A$ may be
regarded as elements of the free $R$-module $R^n$. 

\begin{thm}
We have
  $$ R^n/(v_1,\dots,v_m) \cong (R/\alpha_1 R)\oplus\cdots \oplus
     (R/\alpha_m R). $$
\end{thm}

\begin{proof}
It is easily seen that the allowed row and column operations do not
change the isomorphism class of the quotient of $R^n$ by the rows of
the matrix. Since the conclusion is tautological for diagonal
matrices, the proof follows.
\end{proof}

The quotient module $R^n/(v_1,\dots,v_m)$ is called the
\emph{cokernel} (or sometimes the \emph{Kasteleyn cokernel}) of the
matrix $A$, denoted $\coker(A)$

Recall the basic result from algebra that a finitely-generated module
$M$ over a PID $R$ is a (finite) direct sum of cyclic modules
$R/\alpha_i R$.  Moreover, we can choose the $\alpha_i$'s so that
$\alpha_i|\alpha_{i+1}$ (where $\alpha|0$ for all $\alpha\in R$). In
this case the $\alpha_i$'s are unique up to multiplication by
units. In the case $R=\zz$, this result is the ``fundamental theorem
for finitely-generated abelian groups.'' For a general PID $R$, this
result is equivalent to the PID case of Theorem~\ref{thm:rings}(4). 

\subsection{A formula for SNF}
Recall that a \emph{minor} of a matrix $A$ is the determinant of some
square submatrix. 

\begin{thm} \label{thm:minors}
Let $R$ be a unique factorization domain (e.g., a PID), so that any
two elements have a greatest common divisor (gcd). Suppose that the
$m\times n$ matrix $M$ over $R$ satisfies $M\snf
(\alpha_1,\dots,\alpha_m)$. Then for $1\leq k\leq m$ we have that
$\alpha_1 \alpha_2\cdots \alpha_k$ is equal to the gcd of all $k\times
k$ minors of $A$, with the convention that if all $k\times k$ minors
are 0, then their gcd is 0.
\end{thm}

\noindent \emph{Sketch of proof.} The assertion is easy to check if
$M$ is already in Smith normal form, so we have to show that the
allowed row and column operations preserve the gcd of the $k\times k$
minors. For $k=1$ this is easy. For $k>1$ we can apply the $k=1$ case
to the matrix $\wedge^k M$, the $k$th exterior power of $M$. For
details, see \cite[Prop.~8.1]{m-r}.


\section{The critical group of a graph} \label{sec:cgp}
Let $G$ be a finite graph on the vertex set $V$. We allow multiple
edges but not loops (edges from a vertex to itself). (We could allow
loops, but they turn out to be irrelevant.) Write $\mu(u,v)$ for the
number of edges between vertices $u$ and $v$, and $\deg v$ for the
degree (number of incident edges) of vertex $v$. The \emph{Laplacian
  matrix} $\bm{L}=\bm{L}(G)$ is the matrix with rows and columns
indexed by the elements of $V$ (in some order), with
   $$ \bm{L_{uv}} =\left\{ \begin{array}{rl}
      -\mu(u,v), & \mathrm{if}\ u\neq v\\
      \deg(v), & \mathrm{if}\ u=v. \end{array} \right. $$
The matrix $\bm{L}(G)$ is always singular since its rows sum to
0.\ \  Let $\bm{L_0}=\bm{L_0}(G)$ be $\bm{L}$ with the last row and
column removed. (We can just as well remove any row and column.) The
well-known \emph{Matrix-Tree Theorem} (e.g., \cite[Thm.~5.6.8]{ec2})
asserts that $\det \bm{L_0}=\kappa(G)$, the number of spanning trees
of $G$. Equivalently, if $\#V=n$ and $\bm{L}$ has eigenvalues
$\theta_1,\dots, \theta_n$, where $\theta_n=0$, then $\kappa(G) =
\theta_1\cdots \theta_{n-1}/n$. We are regarding $\bm{L}$ and
$\bm{L_0}$ as matrices over $\zz$, so they both have an SNF. It is
easy to see that $\bm{L_0}\snf (\alpha_1,\dots,\alpha_{n-1})$ if and
only if $\bm{L}\snf (\alpha_1,\dots, \alpha_{n-1},0)$. 

Let $G$ be connected. The group coker$(\bm{L_0})$ has an interesting
interpretation in terms of chip-firing, which we explain below. For
this reason there has been a lot of work on finding the SNF of
Laplacian matrices $\bm{L}(G)$.

A \emph{configuration} is a finite collection $\sigma$ of
indistinguishable chips distributed among the vertices of the graph
$G$. Equivalently, we may regard $\sigma$ as a function $\sigma \colon
V\to \nn=\{0,1,2,\dots\}$. Suppose that for some vertex $v$ we have
$\sigma(v)\geq \deg(v)$. The \emph{toppling} or \emph{firing} $\tau$ of
vertex $v$ is the configuration obtained by sending a chip from $v$ along
each incident edge to the vertex at the other end of the edge. Thus
  $$ \tau(u) = \left\{ \begin{array}{rl}
       \sigma(v)-\deg(v), & u=v\\
       \sigma(u)+\mu(u,v), & u\neq v. \end{array} \right. $$

Now choose a vertex $w$ of $G$ to be a \emph{sink}, and ignore chips
falling into the sink. (We never topple the sink.) This dynamical system
is called the \emph{abelian sandpile} model. A \emph{stable}
configuration is one for which no vertex can topple, i.e.,
$\sigma(v)<\deg(v)$ for all vertices $v\neq w$. It is easy to see that
after finitely many topples a stable configuration will be reached,
which is independent of the order of topples.  (This independence of
order accounts for the word ``abelian'' in ``abelian sandpile.'')

Let $M$ denote the set of all stable configurations.  Define a binary
operation $\oplus$ on $M$ by vertex-wise addition followed by
stabilization. An \emph{ideal} of $M$ is a subset $J\subseteq M$
satisfying $\sigma \oplus J\subseteq J$ for all $\sigma\in M$. The
\emph{sandpile group} or \emph{critical group} $K(G)$ is the minimal
ideal of $M$, i.e., the intersection of all ideals. (Following the
survey \cite{l-p} of Levine and Propp, the reader is encouraged to
prove that the minimal ideal of any finite commutative monoid is a
group.) The group $K(G)$ is independent of the choice of sink up to
isomorphism.

An equivalent but somewhat less abstract definition of $K(G)$ is the
following. A configuration $u$ is called \emph{recurrent} if, for all
configurations $v$, there is a configuration $y$ such that $v\oplus
y=u$. A configuration that is both stable and recurrent is called
\emph{critical}. Given critical configurations $C_1$ and $C_2$,
define $C_1+C_2$ to be the unique critical configuration reachable
from the vertex-wise sum of $C_1$ and $C_2$. 
This operation turns the set of critical configurations into an
abelian group isomorphic to the critical group $K(G)$.

The basic result on $K(G)$ \cite{biggs}\cite{dhar} is the following. 

\begin{thm}
We have $K(G)\cong \coker(\bm{L_0}(G))$. Equivalently, if $\bm{L_0}(G)
\snf (\alpha_1,\dots,\alpha_{n-1})$, then
  $$ K(G) \cong \zz/\alpha_1\zz \oplus \cdots \oplus 
         \zz/\alpha_{n-1} \zz. $$ 
\end{thm}

Note that by the Matrix-Tree Theorem we have $\#K(G)=\det \bm{L_0(G)}
= \kappa(G)$. Thus the critical group $K(G)$ gives a canonical
factorization of $\kappa(G)$. When $\kappa(G)$ has a ``nice''
factorization, it is especially interesting to determine $K(G)$. The
simplest case is $G=K_n$, the complete graph on $n$ vertices. We have
$\kappa(K_n)=n^{n-2}$, a classic result going back to Sylvester and
Borchardt. There is a simple trick for computing $K(K_n)$ based on
Theorem~\ref{thm:minors}. Let $\bm{L_0}(K_n)\snf
(\alpha_1,\dots,\alpha_{n-1})$. Since $\bm{L_0}(K_n)$ has an entry
equal to $-1$, it follows from Theorem~\ref{thm:minors} that
$\alpha_1=1$. Now the $2\times 2$ submatrices (up to row and column
permutations) of $\bm{L_0}(K_n)$ are given by
  $$ \left[ \begin{array}{cc} n-1 & -1\\ -1 & n-1 \end{array} \right],
    \quad
   \left[ \begin{array}{cc} n-1 & -1\\ -1 & -1\end{array} \right],
    \quad
   \left[ \begin{array}{cc} -1 & -1\\ -1 & -1 \\ \end{array} \right],
   $$
with determinants $n(n-2)$, $-n$, and 0. Hence $\alpha_2=n$ by
Theorem~\ref{thm:minors}. Since $\prod \alpha_i = \pm n^{n-2}$ and
$\alpha_i|\alpha_{i+1}$, we get $K(G)\cong (\zz/n\zz)^{n-2}$. 

\textsc{Note.} A similar trick works for the matrix $M=\left[
  \binom{2(i+j)}{i+j}\right]_{i,j=0}^{n-1}$, once it is known that
 $\det M=2^{n-1}$ (e.g., \cite[Thm.~9]{e-r-r}). Every entry of $M$ is
even except for $M_{00}$, so $2|\alpha_2$, yielding $M\snf
(1,2,2,\dots,2)$. The matrix
$\left[\binom{3(i+j)}{i+j}\right]_{i,j=0}^{n-1}$ is much more
  complicated. For instance, when $n=8$ the diagonal elements of the
  SNF are
  $$ 1,\ 3,\ 3,\ 3,\ 3,\ 6,\ 2\cdot 3\cdot 29\cdot 31,\ 2\cdot
  3^2\cdot 11\cdot 
      29\cdot 31\cdot 37\cdot 41. $$
It seems that if $d_n$ denotes the number of diagonal entries of the
SNF that are equal to 3, then $d_n$ is close to $\frac 23n$. The least
$n$ for which $|d_n-\lfloor \frac 23n\rfloor|>1$ is $n=224$. 
For the determinant of $M$, see \cite[(10)]{g-x}. If
$M=\left[ \binom{a(i+j)}{i+j}\right]_{i,j=0}^{n-1}$ for $a\geq 4$,
then $\det M$ does not seem ``nice'' (it doesn't factor into small
factors).

The critical groups of many classes of graphs have been computed. As a
couple of nice examples, we mention threshold graphs (work of
B. Jacobson \cite{jac}) and Paley graphs (D. B. Chandler, P. Sin, and
Q. Xiang \cite{c-s-x}). Critical groups have been generalized in
various ways. In particular, A. M. Duval, C. J. Klivans, and
J. L. Martin \cite{d-k-m} consider the critical group of a simplicial
complex.

\section{Random matrices} \label{sec:random}
There is a huge literature on the distribution of eigenvalues and
eigenvectors of a random matrix. Much less has been done on the
distribution of the SNF of a random matrix. We will restrict our
attention to the situation where $k\geq 0$ and $M$ is an $m\times n$
integer matrix with independent entries uniformly distributed in the
interval $[-k,k]$, in the limit as $k\to\infty$. We write
$P_k^{(m,n)}(\mathcal{E})$ for the probability of some event under
this model (for fixed $k$).  To illustrate that the distribution of
SNF in such a model might be interesting, suppose that $M\snf
(\alpha_1,\dots,\alpha_m)$. Let $j\geq 1$. The probability
$P_k^{(m,n)}(\alpha_1=j)$ that $\alpha_1=j$ is equal to the
probability that $mn$ integers between $-k$ and $k$ have gcd equal to
$j$. It is then a well-known, elementary result that when $mn>1$, 
  \beq \lim_{k\to\infty}P_k^{(m,n)}(\alpha_1=j) =
  \frac{1}{j^{mn}\zeta(mn)}, \label{eq:a1ej} \eeq
where $\zeta$ denotes the Riemann zeta function. This suggests
looking, for instance, at such numbers as
  $$ \lim_{k\to\infty}P_k^{(m,n)}(\alpha_1=1,
    \alpha_2=2, \alpha_3=12). $$ 
In fact, it turns out that if $m<n$ and we specify 
the values $\alpha_1,\dots, \alpha_m$ (subject of course to
$\alpha_1|\alpha_2|\cdots |\alpha_{m-1}$), then the probability as
$k\to\infty$ exists and is strictly between 0 and 1. For $m=n$ the
same is true for specifying $\alpha_1,\dots,\alpha_{n-1}$. However, for
any  $j\geq 1$, we have $\lim_{k\to\infty}P_k^{(n,n)}(\alpha_n=j)=0$. 

The first significant result of this nature is due to Ekedahl
\cite[{\textsection}3]{ekedahl}, namely, let 
  $$ \sigma(n) = \lim_{k\to\infty} P_k^{(n,n)}(\alpha_{n-1}=1). $$
Note that this number is just the probability (as $k\to\infty$) that
the cokernel of the $n\times n$ matrix $M$ is cyclic (has one
generator). Then 
   \beq  \sigma(n) = \frac{\prod_p
  \left(1+\frac{1}{p^2}+\frac{1}{p^3}+\cdots+\frac{1}{p^n}\right)}
   {\zeta(2)\zeta(3)\cdots}, \label{eq:eke1} \eeq
where $p$ ranges over all primes. It is not hard to deduce that
   \bea \lim_{n\to\infty}\sigma(n)  & = &
    \frac{1}{\zeta(6)\prod_{j\geq 4}\zeta(j)} \label{eq:eke}\\[.5em] &
    = & 0.84693590173\cdots. \nonumber\eea
At first sight it seems surprising that this latter probability is not
1.\ \ It is the probability (as $k\to\infty$, $n\to\infty$) that the
$n^2$ $(n-1)\times (n-1)$ minors of $M$ are relatively prime. Thus the
$(n-1)\times (n-1)$ minors do not behave at all like $n^2$ independent
random integers.

Further work on the SNF of random integer matrices appears in
\cite{wood2} and the references cited there. These papers are
concerned with powers of a fixed prime $p$ 
dividing the $\alpha_i$'s. Equivalently, they are working (at least
implicitly) over the $p$-adic integers $\zz_p$. The first paper to
treat systematically SNF over $\zz$ is by Wang and Stanley
\cite{w-s}. One would expect that the behavior of the prime power
divisors to be independent for different primes as $k\to\infty$. This
is indeed the case, though it takes some work to prove. 
In particular, for any positive integers $h\leq m\leq n$ and
$a_1|a_2|\cdots|a_h$ Wang and Stanley determine
 $$ \lim_{k\to\infty}P_k^{(m,n)}(\alpha_1=a_1,\dots,\alpha_h=a_h). $$ 
A typical result is the following:
  \beas \lim_{k\to\infty}P_k^{(n,n)}(\alpha_1=2,\alpha_2=6) & = &
       2^{-n^2}\left(1-\sum_{i=(n-1)^2}^{n(n-1)}2^{-i}+
         \sum_{i=n(n-1)+1}^{n^2-1}2^{-i}\right)\\ & & \cdot
        \frac 32\cdot 3^{-(n-1)^2}(1-3^{(n-1)^2})(1-3^{-n})^2\\
    & & \cdot \prod_{p>3}\left(1-\sum_{i=(n-1)^2}^{n(n-1)}p^{-i}+
         \sum_{i=n(n-1)+1}^{n^2-1}p^{-i}\right). \eeas

A further result in \cite{w-s} is an extension of Ekedahl's formula
\eqref{eq:eke1}. The authors obtain explicit formulas for
 $$ \rho_j(n)\eqdef\lim_{k\to\infty} P_k^{(n,n)}(\alpha_{n-j}=1), $$ 
i.e., the probability (as $k\to\infty$) that the cokernel of $M$ has
at most $j$ generators. Thus \eqref{eq:eke1} is the case
$j=1$. Write $\rho_j=\lim_{n\to\infty} \rho_j(n)$. Numerically we have
  \beas \rho_1 & = & 0.846935901735\\
        \rho_2 & = & 0.994626883543\\
        \rho_3 & = & 0.999953295075\\
        \rho_4 & = & 0.999999903035 \\
        \rho_5 & = &  0.999999999951. \eeas
The convergence $\rho_n\to 1$ looks very rapid. In fact
\cite[(4.38)]{w-s}, 
  $$ \rho_n = 1 -
         c\,2^{-(n+1)^2}(1-2^{-n}+O(4^{-n})), $$ 
where
  $$ c = \frac{1}{(1-\frac 12)(1-\frac 14)(1-\frac 18)\cdots}
       = 3.46275\cdots. $$

A major current topic related to eigenvalues and eigenvectors of
random matrices is \emph{universality} (e.g., \cite{t-v}). A
certain distribution of eigenvalues (say) occurs for a large class of
probability distributions on the matrices, not just for a special
distribution like the GUE model on the space of $n\times n$ Hermitian
matrices. Universality of SNF over the rings $\zz_p$ of $p$-adic
integers and over $\zz/n\zz$ was considered by Kenneth Maples
\cite{maples}. 
On the other hand, Clancy, Kaplan, Leake, Payne and Wood \cite{cklpw}
make some conjectures for the SNF distribution of the Laplacian matrix
of an Erd\H{o}s-R\'enyi random graph that differs from the
distribution obtained in \cite{w-s}. (It is clear, for instance, that
$\alpha_1=1$ for Laplacian matrices, in contradistinction to
equation~\eqref{eq:a1ej}, but conceivably equation~\eqref{eq:eke} could
carry over.)  Some progress on these conjectures was made by Wood
\cite{wood}.

\section{Symmetric functions}
\subsection{An up-down linear transformation}
Many interesting matrices arise in the theory of symmetric
functions. We will adhere to notation and terminology on this subject
from \cite[Chap.~7]{ec2}. For our first example, let $\Lambda_\qq^n$
denote the $\qq$-vector space of homogeneous symmetric functions
of degree $n$ in the variables $x=(x_1,x_2,\dots)$ with rational 
coefficients. One basis for $\Lambda_\qq^n$ consists of the
\emph{Schur functions} $s_\lambda$ for $\lambda\vdash n$. Define a
linear transformation $\psi_n\colon\Lambda_\qq^n\to\Lambda_\qq^n$
by
 $$ \psi_n(f) = \frac{\partial}{\partial p_1}p_1f. $$ 
Here $p_1=s_1=\sum x_i$, the first power sum symmetric function. The
notation $\frac{\partial}{\partial p_1}$ indicates that we
differentiate with respect to $p_1$ after writing the argument as a
polynomial in the $p_k$'s, where $p_k=\sum x_i^k$. It is a standard
result \cite[Thm.~7.15.7, Cor.~7.15.9, Exer.~7.35]{ec2} that
for $\lambda\vdash n$,
   \beas p_1 s_\lambda & = & \sum_{\substack{\mu\vdash
       n+1\\ \mu\supset\lambda}}s_\mu\\
      \frac{\partial}{\partial p_1}s_\lambda & = & s_{\lambda/1}\ = \
      \sum_{\substack{\mu\vdash  n-1\\ \mu\subset\lambda}} s_\mu. \eeas
Note that the power sum $p_\lambda$, $\lambda\vdash n$, is an
eigenvector for $\psi_n$ with eigenvalue
$m_1(\lambda)+1$, where $m_1(\lambda)$ is the number of 1's in
$\lambda$. Hence
  $$ \det \psi_n = \prod_{\lambda\vdash n}
             (m_1(\lambda)+1). $$ 
The factorization of $\det\psi_n$ suggests looking at the SNF of
$\psi_n$ with respect to the basis $\{s_\lambda\}$. We denote this
matrix by $[\psi_n]$. Since the matrix transforming the $s_\lambda$'s
to the $p_\mu$'s is not invertible over $\zz$, we cannot simply
convert the diagonal matrix with entries $m_1(\lambda)+1$ to SNF.  As
a special case of a more general conjecture Miller and Reiner
\cite{m-r} conjectured the SNF of $[\psi_n]$, which was then proved by
Cai and Stanley \cite{c-s}. Subsequently Nie \cite{nie} and Shah
\cite{shah} made some further progress on the conjecture of Miller and
Reiner.  We state two equivalent forms of the result of Cai and
Stanley.

\begin{thm} \label{thm:cai}
 Let $[\psi_n]\snf (\alpha_1,\dots,\alpha_{p(n)})$, where $p(n)$
 denotes the number of partitions of $n$.
\be\item[(a)] The $\alpha_i$'s are as follows:
  \begin{itemize}
    \item $(n+1)(n-1)!$, with multiplicity $1$
    \item $(n-k)!$, with multiplicity $p(k+1)-2p(k)+p(k-1)$, $3\leq
      k\leq n-2$
    \item $1$, with multiplicity $p(n)-p(n-1)+p(n-2)$. 
   \end{itemize}
 \item[(b)] Let $\mathcal{M}_1(n)$ be the multiset of all numbers
   $m_1(\lambda)+1$, for $\lambda\vdash n$. Then $\alpha_{p(n)}$ is the
   product of the \emph{distinct} elements of $\mathcal{M}_1(n)$;
   $\alpha_{p(n)-1}$ is the product of the remaining \emph{distinct}
   elements of $\mathcal{M}_1(n)$, etc.
\ee
\end{thm}

In fact, the following stronger result than Theorem~\ref{thm:cai} is
actually proved. 

\begin{thm} \label{thm:cai2}
Let $t$ be an indeterminate. Then the matrix $[\psi_n+tI]$ 
has an SNF over $\zz[t]$. 
\end{thm}

To see that Theorem~\ref{thm:cai2} implies Theorem~\ref{thm:cai}, use
the fact that $[\psi_n]$ is a symmetric matrix (and therefore
semisimple), and for each eigenvalue $\lambda$ of $\psi_n$ consider
the rank of the matrices obtained by substituting $t=-\lambda$ in
$[\psi_n+tI]$ and its SNF over $\zz[t]$. For details and further
aspects, see \cite[{\textsection}8.2]{m-r}.

The proof of Theorem \ref{thm:cai2} begins by working with the basis
$\{h_\lambda\}$ of complete symmetric functions rather than with the
Schur functions, which we can do since the transition matrix between
these bases is an integer unimodular matrix. The proof then consists
basically of describing the row and column operations to achieve SNF. 

The paper \cite{c-s} contains a conjectured generalization of 
Theorem~\ref{thm:cai2} to the operator $\psi_{n,k}\eqdef
k\frac{\partial}{\partial 
  p_k}p_k\colon \Lambda_\qq^n\to \Lambda_\qq^n$ for any $k\geq 1$.
Namely, the matrix $[\psi_{n,k}+tI]$ with 
respect to the basis $\{s_\lambda\}$ has an SNF over $\zz[t]$. This
implies that if $[\psi_{n,k}]\snf (\alpha_1,\dots,\alpha_{p(n)})$ and
$\mathcal{M}_k(n)$ denotes the multiset of all numbers
$k(m_k(\lambda)+1)$, for $\lambda\vdash n$, then $\alpha_{p(n)}$ is
the product of the \emph{distinct} elements of $\mathcal{M}_k(n)$;
$\alpha_{p(n)-1}$ is the product of the remaining \emph{distinct}
elements of $\mathcal{M}_k(n)$, etc. This conjecture was proved in
2015 by Zipei Nie (private communication).

There is a natural generalization of the SNF of $\psi_{n,k}$, namely,
we can look at operators like $(\prod
\lambda_i)\frac{\partial^\ell}{\partial p_\lambda}p_\lambda$. Here
$\lambda$ is a partition of $n$ with $\ell$ parts and
  $$ \frac{\partial^\ell}{\partial p_\lambda} =
   \frac{\partial^\ell}{\partial p_1^{m_1}\partial p_2^{m_2}\cdots}, $$
where $\lambda$ has $m_i$ parts equal to $i$. Even more generally, if
$\lambda, \mu\vdash n$ where $\lambda$ has $\ell$ parts, then we could
consider $(\prod \lambda_i)\frac{\partial^\ell}{\partial
  p_\lambda}p_\mu$. No conjecture is known for the SNF (with respect
to an integral basis), even when $\lambda=\mu$.

\subsection{A specialized Jacobi-Trudi matrix}
A fundamental identity in the theory of symmetric functions is the
\emph{Jacobi-Trudi identity}. Namely, if $\lambda$ is a partition with
at most $t$ parts, then the \emph{Jacobi-Trudi
  matrix} $\jtl$ is defined by
  $$ \jtl=\left[ h_{\lambda_i+j-i}\right]_{i,j=1}^t, $$
where $h_i$ denotes the complete symmetric function of degree $i$
(with $h_0=1$ and $h_{-i}=0$ for $i\geq 1$). 
The \emph{Jacobi-Trudi identity} \cite[{\textsection}7.16]{ec2}
asserts that 
$\det \jtl =s_\lambda$, the Schur function indexed by $\lambda$.

For a symmetric function $f$, let $\varphi_n f$ denote the
specialization $f(1^n)$, that is, set $x_1=\cdots=x_n=1$ and all other
$x_i=0$ in $f$. It is easy to see \cite[Prop.~7.8.3]{ec2} that
   \beq \varphi_n h_i =\binom{n+i-1}{i}, \label{eq:phihi} \eeq
a polynomial in $n$ of degree $i$. Identify $\lambda$ with its (Young)
diagram, so the squares of $\lambda$ are indexed by pairs $(i,j)$,
$1\leq i\leq \ell(\lambda)$, $1\leq j\leq \lambda_i$. The
\emph{content} $c(u)$ of the square $u=(i,j)$ is defined to be
$c(u)=j-i$. A standard result \cite[Cor.~7.21.4]{ec2} in the theory of
symmetric functions states that
  \beq \varphi_n s_\lambda =\frac{1}{H_\lambda}\prod_{u\in\lambda}
   (n+c(u)), \label{eq:hc} \eeq
where $H_\lambda$ is a positive integer whose value is irrelevant here
(since it is a unit in $\qq[n]$). Since this polynomial factors a lot
(in fact, into linear factors) over $\qq[n]$, we are motivated to
consider the SNF of the matrix 
  $$ \varphi_n\jtl = \left[ \binom{n+\lambda_i+j-i-1}{\lambda_i+j-i}
     \right]_{i,j=1}^t. $$
Let $D_k$ denote the $k$th \emph{diagonal hook} of $\lambda$, i.e.,
all squares $(i,j)\in\lambda$ such that either $i=k$ and $j\geq k$, or
$j=k$ and $i\geq k$. Note that $\lambda$ is a disjoint union of its
diagonal hooks. If $r=\mathrm{rank}(\lambda)\mathrel{\mathop:}= \max\{
i\st \lambda_i\geq i\}$, then note also that $D_k=\emptyset$ for
$k>r$. The following result was proved in \cite{rs:jt}.

\begin{thm} \label{thm:jt} 
Let $\varphi_n\jtl\snf (\alpha_1,\alpha_2,\dots,\alpha_t)$, where
$t\geq \ell(\lambda)$. Then we can take
  $$ \alpha_i = \prod_{u\in D_{t-i+1}} (n+c(u)). $$
\end{thm}

An equivalent statement to Theorem~\ref{thm:jt} is that the $\alpha_i$'s
are squarefree (as polynomials in $n$), since $\alpha_t$ is the
largest squarefree factor of $\varphi_n s_\lambda$, $\alpha_{t-1}$ is
the largest squarefree factor of $(\varphi_n s_\lambda)/\alpha_t$, etc.

\begin{ex} 
Let $\lambda=(7,5,5,2)$. Figure~\ref{fig1} shows the diagram of
$\lambda$ with the content of each square. Let $t=\ell(\lambda)=4$. We
see that 
  \beas \alpha_4 & = & (n-3)(n-2)\cdots (n+6)\\
        \alpha_3 & = & (n-2)(n-1)n(n+1)(n+2)(n+3)\\
        \alpha_2 & = & n(n+1)(n+2)\\
        \alpha_1 & = & 1. \eeas

\begin{figure}
\centering
\centerline{\includegraphics[width=6cm]{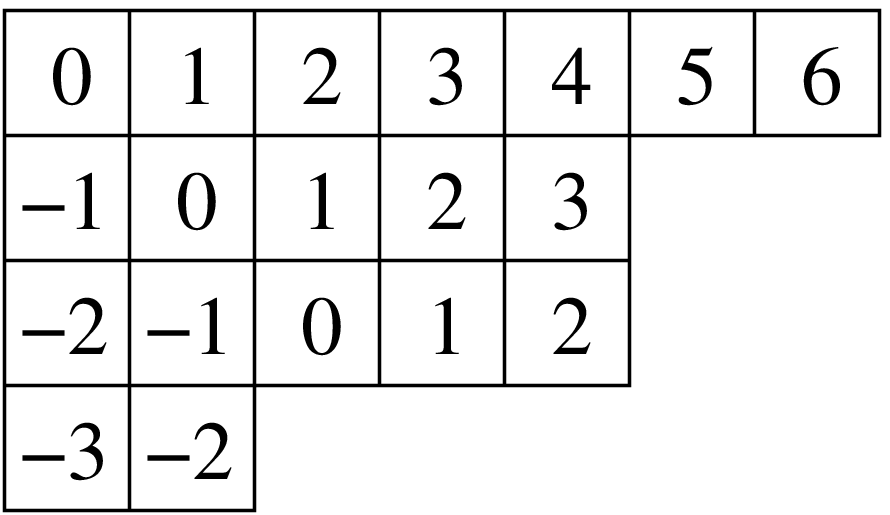}}
\caption{The contents of the partition $(7,5,5,2)$}
\label{fig1}
\end{figure}

\end{ex}

The problem of computing the SNF of a suitably specialized
Jacobi-Trudi matrix was raised by Kuperberg \cite{kup}. His Theorem~14
has some overlap with our Theorem~\ref{thm:jt}. Propp
\cite[Problem~5]{propp} mentions a two-part question of Kuperberg. The
first part is equivalent to our Theorem~\ref{thm:jt} for rectangular
shapes. (The second part asks for an interpretation in terms of
tilings, which we do not consider.)
 
Theorem~\ref{thm:jt} is proved not by the more usual method of row and
column operations. Rather, the gcd of the $k\times k$ minors is
computed explicitly so that Theorem~\ref{thm:minors} can be
applied. Let $M_k$ be the bottom-left $k\times k$ submatrix of
$\jtl$. Then $M_k$ is itself the Jacobi-Trudi matrix of a certain
partition $\mu^k$, so $\varphi_n M_k$ can be explicitly evaluated. One
then shows using the Littlewood-Richardson rule that every $k\times k$
minor of $\varphi_n\jtl$ is divisible by $\varphi_n M_k$. Hence
$\varphi_n M_k$ is the gcd of the $k\times k$ minors of
$\varphi_n\jtl$, after which the proof is a routine computation.

There is a natural $q$-analogue of the specialization $f(x) \to
f(1^n)$, namely, $f(x)\to f(1,q,q^2,\dots,q^{n-1})$. Thus we can ask
for a $q$-analogue of Theorem~\ref{thm:jt}. This can be done using the
same proof technique, but some care must be taken in order to get a
$q$-analogue that reduces directly to Theorem~\ref{thm:jt} by setting
$q=1$. When this is done we get the following result
\cite[Thm.~3.2]{rs:jt}. 

\begin{thm} \label{thm:jtq}
For $k\geq 1$ let
  $$ f(k) =\frac{n(n+\boldsymbol{(1)})(n+\boldsymbol{(2)})\cdots 
     (n+\boldsymbol{(k-1)})}{\boldsymbol{(1)}\boldsymbol{(2)}
    \cdots \boldsymbol{(k)}}, $$
where $\boldsymbol{(j)}=(1-q^j)/(1-q)$ for any $j\in\zz$.
Set $f(0)=1$ and $f(k)=0$ for $k<0$. Define 
  $$ \jtq=\left[f(\lambda_i-i+j)\right]_{i,j=1}^t, $$
where $\ell(\lambda)\leq t$. Let $\jtq\snf
(\gamma_1,\gamma_2,\dots,\gamma_t)$ over the ring 
$\qq(q)[n]$. Then we can take
  $$ \gamma_i = \prod_{u\in D_{t-i+1}}(n+\boldsymbol{c(u)}). $$
\end{thm}

\section{A multivariate example}
In this section we give an example where the SNF exists over a
multivariate polynomial ring over $\zz$. Let $\lambda$ be a partition,
identified with its Young diagram regarded as a set of squares; we fix
$\lambda$ for all that follows.  Adjoin to $\lambda$ a border strip
extending from the end of the first row to the end of the first column
of $\lambda$, yielding an \emph{extended partition} $\lambda^*$. Let
$(r,s)$ denote the square in the $r$th row and $s$th column of
$\lambda^*$. If $(r,s)\in\lambda^*$, then let $\lambda(r,s)$ be the
partition whose diagram consists of all squares $(u,v)$ of $\lambda$
satisfying $u\geq r$ and $v\geq s$. Thus $\lambda(1,1)=\lambda$, while
$\lambda(r,s)=\emptyset$ (the empty partition) if
$(r,s)\in\lambda^*\setminus\lambda$. Associate with the square $(i,j)$
of $\lambda$ an indeterminate~$x_{ij}$. Now for each square $(r,s)$
of~$\lambda^*$, associate a polynomial $P_{rs}$ in the
variables~$x_{ij}$, defined as follows:
\beq
P_{rs} = \sum_{\mu\subseteq
  \lambda(r,s)}\prod_{(i,j)\in\lambda(r,s)\setminus \mu} x_{ij},
  \label{eq:prsdef}
\eeq
where $\mu$ runs over all partitions contained in $\la (r,s)$.
In particular, if $(r,s)\in\lambda^*\setminus \lambda$ then $P_{rs}=1$.
Thus for $(r,s)\in\lambda$, $P_{rs}$ may be regarded as a generating
function for the squares
of all skew diagrams $\lambda(r,s)\setminus \mu$.  For instance, if
$\lambda=(3,2)$ and we set $x_{11}=a$, $x_{12}=b$, $x_{13}=c$,
$x_{21}=d$, and $x_{22}=e$, then Figure~\ref{fig2} shows the extended
diagram $\lambda^*$ with the polynomial $P_{rs}$ placed in the
square~$(r,s)$.

\begin{figure}
\centerline{\includegraphics[width=8cm]{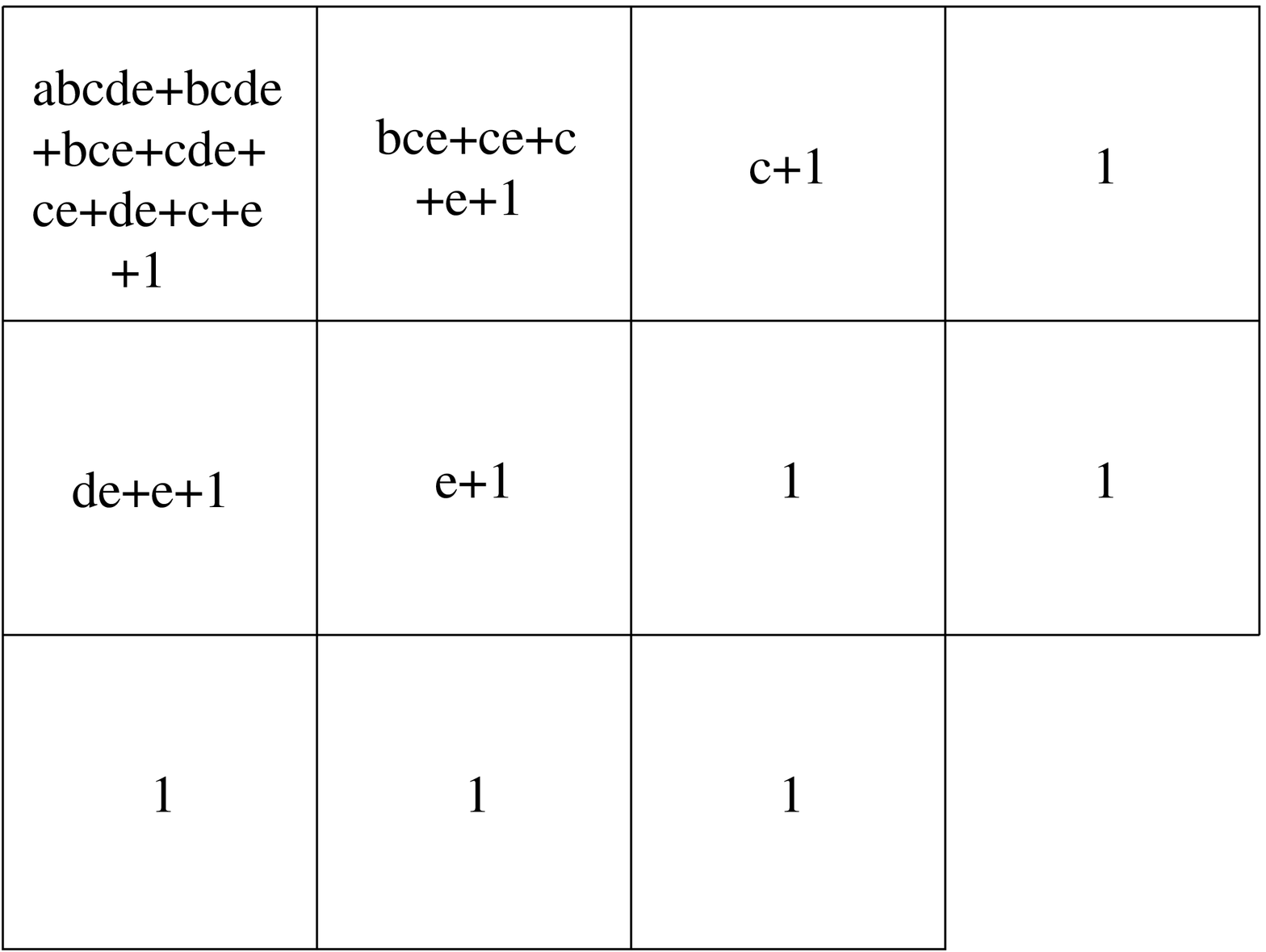}}
\caption{The polynomials $P_{rs}$ for $\lambda=(3,2)$}
\label{fig2}
\end{figure}

Write
  $$ A_{rs}=\prod_{(i,j)\in\lambda(r,s)} x_{ij}. $$
Note that $A_{rs}$ is simply the leading term of $P_{rs}$.
Thus for $\lambda=(3,2)$ as in Figure~\ref{fig2} we have
$A_{11}=abcde, A_{12}=bce$, $A_{13}=c$, $A_{21}=de$, and
$A_{22}=e$.

For each square $(i,j)\in\lambda^*$ there will be a unique subset of
the squares of $\lambda^*$ forming an $m\times m$ square $S(i,j)$ for
some $m\geq 1$, such that the upper left-hand corner of $S(i,j)$ is
$(i,j)$, and the lower right-hand corner of $S(i,j)$ lies in
$\lambda^*\setminus \lambda$. In fact, if $\rho_{ij}$ denotes the
\emph{rank} of $\lambda(i,j)$ (the number of squares on the main
diagonal, or equivalently, the largest $k$ for which
$\lambda(i,j)_k\geq k$), then $m=\rho_{ij}+1$. Let $M(i,j)$ denote the
matrix obtained by inserting in each square $(r,s)$ of $S(i,j)$ the
polynomial $P_{rs}$. For instance, for the partition $\lambda=(3,2)$
of Figure~\ref{fig2}, the matrix $M(1,1)$ is given by
 $$ M(1,1) = \left[ \begin{array}{ccc}
  P_{11} & bce+ce+c+e+1 & c+1\\ de+e+1 & e+1 & 1\\
  1 & 1 & 1 \end{array} \right], $$
where $P_{11}=abcde+bcde+bce+cde+ce+de+c+e+1$. Note that for this
example we have
  $$ \det M(1,1)= A_{11}A_{22}A_{33}=abcde\cdot e\cdot 1=abcde^2. $$

The main result on the matrices $M(i,j)$ is the following. For
convenience we state it only for $M(1,1)$, but it applies to any
$M(i,j)$ by replacing $\lambda$ with $\lambda(i,j)$. 

\begin{thm} \label{thm:bessen}
Let $\rho=\mathrm{rank}(\lambda)$. The matrix $M(1,1)$ has an SNF over
$\zz[x_{ij}]$, given explicitly by   
   $$ M(1,1)\snf
      (A_{11},A_{22},\dots,A_{\rho+1,\rho+1}). $$  
Hence $\det M(1,1)=A_{11}A_{22}\cdots A_{\rho\rho}$ (since
$A_{\rho+1,\rho+1}=1$). 
\end{thm}

Theorem~\ref{thm:bessen} is proved by finding row and column
operations converting $M(1,1)$ to SNF. In \cite{b-s} this is done in
two ways: an explicit description of the row and column operations,
and a proof by induction that such operations exist without stating
them explicitly. 

Another way to describe the SNF of $M(1,1)$ is to replace its
nondiagonal entries with 0 and a diagonal entry with its leading term
(unique monomial of highest degree). Is there some conceptual reason
why the SNF has this simple description?

If we set each $x_{ij}=1$ in $M(1,1)$ then we get $\det
M(1,1)=1$. This formula is equivalent to result of Carlitz, Roselle,
and Scoville \cite{c-r-s} which answers a question posed by Berlekamp
\cite{berl1}\cite{berl2}. If we set each $x_{ij}=q$ in $M(1,1)$ and
take $\lambda=(m-1,m-2,\dots,1)$, then the entries of $M(1,1)$ are
certain $q$-Catalan numbers, and $\det M(1,1)$ was determined by
Cigler \cite{cigler1}\cite{cigler2}. This determinant (and some
related ones) was a primary motivation for \cite{b-s}. Miller and
Stanton \cite{m-r} have generalized the $q$-Catalan result to Hankel
matrices of moments of orthogonal polynomials and some other similar
matrices. 

Di Francesco \cite{difran} shows that the polynomials $P_{rs}$ satisfy
the ``octahedron recurrence'' and are related to cluster algebras,
integrable systems, dimer models, and other topics. 

\section{The Varchenko matrix}
Let $\ca$ be a finite arrangement (set) of affine hyperplanes in
$\rr^n$. The complement $\rr^n-\bigcup_{H\in\ca}H$ consists of a
disjoint union of finitely many open \emph{regions}. Let $\calr(\ca)$
denote the 
set of all regions. For each hyperplane $H\in\ca$ associate an
indeterminate $a_H$.  If $R,R'\in \calr(\ca)$ then let sep$(R,R')$
denote the set of $H\in\ca$ separating $R$ from $R'$, that is, $R$ and
$R'$ lie on different sides of $H$. Now define a matrix $V(\ca)$ as
follows. The rows and columns are indexed by $\calr(\ca)$ (in some
order). The $(R,R')$-entry is given by
  $$ V_{RR'} = \prod_{H\in \mathrm{sep}(R,R')} a_H. $$
If $x$ is any nonempty intersection of a set of hyperplanes
in $\ca$, then define $a_x=\prod_{H\supseteq x}a_H$.   Varchenko
\cite{varchenko} showed that
  \beq \det V(\ca) = \prod_x (1-a_x^2)^{n(x)p(x)}, \label{eq:vardet}
  \eeq
for certain nonnegative integers $n(x),p(x)$ which we will not define
here. 

\textsc{Note.} We include the intersection $x$ over the empty set of
hyperplanes, which is the ambient space $\rr^n$.  This gives an
irrelevant factor of 1 in the determinant above, but it also accounts
for an essential diagonal entry of 1 in Theorem~\ref{thm:g-z} below.

Since $\det V(\ca)$ has such a nice factorization, it is natural to ask
about its diagonal form or SNF. Since we are working over the
polynomial ring $\zz[a_H\st H\in\ca]$ or $\qq[a_H\st H\in\ca]$, there
is no reason for a diagonal form to exist. Gao and 
Zhang \cite{g-z} found the condition for this property to hold. We say
that $\ca$ is \emph{semigeneric} or \emph{in semigeneral form} if for
any $k$ hyperplanes $H_1,\dots,H_k\in\ca$ with intersection
$x=\bigcap_{i=1}^k H_i$, either codim$(x)=k$ or $x=\emptyset$. (Note
that $x$ is an affine subspace of $\rr^n$ so has a well-defined
codimension.) In particular, $x=\emptyset$ if $k>n$.

\begin{thm} \label{thm:g-z}
The matrix $V(\ca)$ has a diagonal form if and only if $\ca$ is
semigeneric. In this case, the diagonal entries of $\ca$ are given by 
$\prod_{H\supseteq x}(1-a_H^2)$, where $x$ is a nonempty intersection
of the hyperplanes in some subset of $\ca$.
\end{thm}

Gao and Zhang actually prove their result for pseudosphere
arrangements, which are a generalization of hyperplane
arrangements. Pseudosphere arrangements correspond to oriented
matroids. 

\begin{ex}
Let $\ca$ be the arrangement of three lines in $\rr^2$ shown in
Figure~\ref{fig:arrex}, with the hyperplane variables $a,b,c$ as in
the figure. This arrangement is semigeneric. The diagonal entries
of the diagonal form of $V(\ca)$ 
are 
  $$ 1,\ 1-a^2,\ 1-b^2,\ 1-c^2, (1-a^2)(1-c^2),\ (1-b^2)(1-c^2). $$
\end{ex}

\begin{figure}
\centerline{\includegraphics[width=6cm]{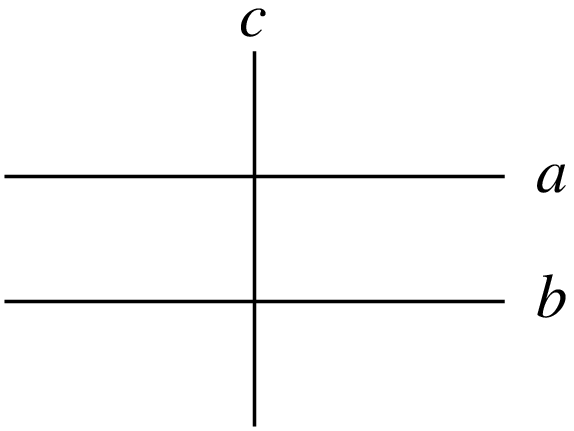}}
\caption{An arrangement of three lines               }
\label{fig:arrex}
\end{figure}

Now define the \emph{$q$-Varchenko matrix} $V_q(\ca)$ of $\ca$ to be
the result of substituting $a_H=q$ for all $H\in\ca$. Equivalently,
$V_q(\ca)_{RR'} = q^{\#\mathrm{sep}(R,R')}$. The SNF of $V_q(\ca)$
exists over the PID $\qq[q]$, and it seems to be a very interesting
and little studied problem to determine this SNF. Some special cases
were determined by Cai and Mu \cite{c-s}. A generalization related to
distance matrices of graphs was considered by Shiu \cite{shiu}. Note
that by equation~\eqref{eq:vardet} the diagonal entries of the SNF of
$V_q(\ca)$ will be products of cyclotomic polynomials $\Phi_d(q)$.

  The main paper to date on the SNF of $V_q(\ca)$ is by Denham and
  Hanlon \cite{d-h}. In particular, let 
  $$ \chi_\ca(t)=\sum_{i=0}^n (-1)^i c_i t^{n-i} $$ 
  be the \emph{characteristic polynomial} of
  $\ca$, as defined for instance in \cite[\textsection
    1.3]{rs:hyp}\cite[\textsection 3.11.2]{ec1}.  
  Denham and Hanlon show the following in their Theorem~3.1.

  \begin{thm}
 Let $N_{d,i}$ be the number of diagonal entries of the SNF of
 $V_q(\ca)$ that are exactly divisible by $\Phi_d(q)^i$. Then $N_{1,i}= 
  c_i$. 
  \end{thm}
  
It is easy to see that $N_{1,i}=N_{2,i}$. Thus the next step would be
to determine $N_{3,i}$ and $N_{4,i}$. 

An especially interesting hyperplane arrangement is the \emph{braid
  arrangement} $\cb_n$ in $\rr^n$, with hyperplanes $x_i=x_j$ for
$1\leq i<j\leq n$. The determinant of $V_q(\cb_n)$, originally due to
Zagier \cite{zagier}, is given by
  $$ \det V_q(\cb_n) = \prod_{j=2}^n\left( 1-q^{j(j-1)}\right)
       ^{\binom nj(j-2)!\,(n-j+1)!}. $$
An equivalent description of $V_q(\cb_n)$ is the following. Let $\sn$
denote the symmetric group of all permutations of $1,2,\dots,n$, and
let inv$(w)$ denote the number of inversions of $w\in\sn$, i.e.,
inv$(w) =\#\{(i,j)\st 1\leq i<j\leq n,\ w(i)>w(j)\}$. Define
$\Gamma_n(q) =\sum_{w\in\sn}q^{\mathrm{inv}(w)}w$, an element of the
group algebra $\qq[q]\sn$. The element $\Gamma_n(q)$ acts on
$\qq[q]\sn$ by left multiplication, and $V_q(\cb_n)$ is the matrix of
this linear transformation (with a suitable indexing of rows and
columns) with respect to the basis $\sn$. The SNF of $V_q(\cb_n)$
(over the PID $\qq[q]$) is not known. Denham and Hanlon
\cite[\textsection 5]{d-h} compute it for $n\leq 6$.

Some simple representation theory allows us to refine the SNF of
$V_q(\cb_n)$. The complex irreducible representations
$\varphi_\lambda$ of $\sn$ are indexed by partitions $\lambda\vdash
n$. Let $f^\lambda=\dim \varphi_\lambda$. The action of $\sn$ on
$\qq\sn$ by right multiplication commutes with the action of
$\Gamma_n(q)$. It follows (since every irreducible representation of
$\sn$ can be defined over $\zz$) that by a unimodular change of basis
we can write
  $$ V_q(\cb_n) = \bigoplus_{\lambda\vdash n} f^\lambda V_\lambda, $$
for some integral matrices $V_\lambda$ of size $f^\lambda\times
f^\lambda$. Thus computing $\det V_\lambda$ and the SNF of $V_\lambda$
is a refinement of computing $\det V_q(\cb_n)$ and the SNF of
$V_q(\cb_n)$. (Computing the SNF of each $V_\lambda$ would give a
diagonal form of $V_q(\cb_n)$, from which it is easy to determine the
SNF.) The problem of computing $\det V_\lambda$ was solved by Hanlon
and Stanley \cite[Conj.~3.7]{h-s}. Of course the SNF of $V_\lambda$
remains open since the same is true for $V_q(\cb_n)$. Denham and
Hanlon have computed the SNF of $V_\lambda$ for
$\lambda\vdash n\leq 6$ and published the results for $n\leq 4$ in
\cite[\textsection 5]{d-h}. For instance, for the partitions
$\lambda\vdash 4$ we have the following diagonal elements of the SNF
of $V_\lambda$: 
  $$ \begin{array}{rl}
     (4): & \Phi_2^2\Phi_3\Phi_4\\
     (3,1): & \Phi_1\Phi_2,\ \Phi_1^2\Phi_2^2\Phi_3,\ \Phi_1^3\Phi_2^3
             \Phi_3^2\\
     (2,2): & \Phi_1^2\Phi_2^2,\ \Phi_1^2\Phi_2^2\Phi_{12}\\
     (2,1,1): & \Phi_1\Phi_2,\ \Phi_1^2\Phi_2^2\Phi_6,\
              \Phi_1^3\Phi_2^3\Phi_6^2\\ 
     (1,1,1,1): & \Phi_1\Phi_2\Phi_4\Phi_6, \end{array} $$   
where $\Phi_d$ denotes the cyclotomic polynomial whose zeros are the
primitive $d$th roots of unity. For a nice survey of this topic see
Denham and Hanlon \cite{d-h2}. 

The discussion above of $\Gamma_n$ suggests that it might be
interesting to consider the SNF of other elements of $R\sn$ for
suitable rings $R$ (or possibly $RG$ for other finite groups $G$). One
intriguing example is the \emph{Jucys-Murphy element} (though it first
appears in the work of Alfred Young \cite[\textsection 19]{young})
$X_k\in\qq \sn$, $1\leq k\leq n$. It is defined by $X_1=0$ and
  $$ X_k = (1,k)+(2,k)+\cdots+(k-1,k),\ \ 2\leq k\leq n, $$ 
where $(i,k)$ denotes the transposition interchanging $i$ and $k$.
Just as for $\Gamma_n(q)$, we can choose an integral basis for
$\qq\sn$ (that is, a $\zz$-basis for $\zz\sn$) so that the action of
$X_k$ on $\qq\sn$ with respect to this basis has a matrix of the form
$\bigoplus_{\lambda\vdash n} f^\lambda W_{\lambda,k}$.  The
eigenvalues of $W_{\lambda,k}$ are known to be the contents of the
positions occupied by $k$ in all standard Young tableaux of shape
$\lambda$. For instance, when $\lambda=(5,1)$ the standard Young
tableaux are
  $$ \begin{array}{lllll} 
     1\,2\,3\,4\,5 & 1\,2\,3\,4\,6 & 1\,2\,3\,5\,6 & 1\,2\,4\,5\,6 &
      1\,3\,4\,5\,6\\ 6 & 5 & 4 & 3 & 2
     \end{array}. $$
The positions occupied by 5 are $(1,5)$, $(2,1)$, $(1,4)$, $(1,4)$,
$(1,4)$. Hence the eigenvalues of $W_{(5,1),5}$ are $5-1=4$, 
  $1-2=-1$, and $4-1=3$ (three times). Darij Grinberg (private
  communication) computed the SNF of the matrices $W_{\lambda,k}$ for
  $\lambda\vdash n\leq 7$. On the basis of this data we make the
  following conjecture.

\begin{conj}
Let $\lambda\vdash n$, $1\leq k\leq n$, and $W_{\lambda,k}\snf
(\alpha_1,\dots,\alpha_{f^\lambda}) $. Fix $1\leq r\leq f^\lambda$. Let
$S_r$ be the set of positions $(i,j)$ that $k$ occupies in at least
$r$ of the SYT's of shape $\lambda$. 
Then $\alpha_{f^\lambda-r+1}=\pm\prod_{(i,j)\in S_r}(j-i)$.   
\end{conj} 

Note in particular that every SNF diagonal entry is (conjecturally) a
product of some of the eigenvalues of $W_{\lambda,k}$.

For example, when $\lambda=(5,1)$ and $k=5$ we have $f^{(5,1)}=5$ and 
$S_1=\{(1,5),(2,1),(1,4)\}$, $S_2=S_3=\{(1,4)\}$, $S_4=S_5=
\emptyset$. Hence $W_{(5,1),5}\snf (1,1,3,3,12)$.

\end{document}